\documentclass[reqno]{amsart}
\usepackage{stix2}
\usepackage{amssymb}

\newtheorem{theorem}{Theorem}
\newtheorem{lemma}[theorem]{Lemma}
\newtheorem{proposition}[theorem]{Proposition}
\newtheorem{corollary}[theorem]{Corollary}

\newtheorem{remark}[theorem]{Remark}

\newtheorem*{thma}{Theorem A}
\newtheorem*{thmb}{Theorem B}

\newcommand{\R}{{\mathbb R}}
\newcommand{\C}{{\mathbb C}}

\newcommand{\D}{{\mathbb D}}
\newcommand{\inr}{\int_{\R}}
\newcommand{\inc}{\int_{\C}}

\DeclareMathOperator{\re}{{\rm Re}\,}
\DeclareMathOperator{\im}{{\rm Im}\,}
\DeclareMathOperator{\tr}{{\rm tr}\,}

\begin{document}

\title{Canonical integral operators on the Fock space II}

\author{Xingtang Dong}
\address{School of Mathematics and KL-AAGDM, Tianjin University, Tianjin 300354, China.}
\email{dongxingtang@163.com; dongxingtang@tju.edu.cn}

\author{Kehe Zhu}
\address{Department of Mathematics and Statistics, SUNY, Albany, NY 12222, USA.}
\email{kzhu@math.albany.edu}



\thanks{\noindent Research supported by NNSF of China (Grant Numbers 
12271396 and 12271328).}

\begin{abstract}
In \cite{DZ3} we introduced and studied a two-parameter family of integral
operators $T^{(s,t)}$ on the Fock space $F^2$ of the complex plane. Under
the inverse Bargmann transform, these operators include the classical {\it linear
canonical transforms} in mathematical physics as special cases, so we called
$T^{(s,t)}$ {\it canonical linear operators} on the Fock space. In this paper
we continue the study of these operators. We show that when a canonical linear
operator $T^{(s,t)}$ is compact, it actually belongs to the Schatten class $S_p$
for all $p>0$. In this case, we find all singular values, determine the $S_p$ 
norm, and obtain a trace formula for $T^{(s,t)}$. We also show that the 
boundedness (and a natural version of compactness) of $T^{(s,t)}$ on $F^p$ 
for any given $p\in(0,\infty]$ is equivalent to the boundedness (and 
compactness) of $T^{(s,t)}$ on $F^2$. Our analysis is based on estimates 
and computations with the integral kernel of $T^{(s,t)}$, which also yield 
some interesting results about the Berezin transform and the bivariate 
Berezin transform of $T^{(s,t)}$.
\end{abstract}

\maketitle

\section{Introduction}

Let $\C$ be the complex plane and $H(\C)$ be the linear space of all entire
functions on $\C$. For $0<p\leq\infty$ we use $F^p$ to denote the space of all
$f\in H(\C)$ such that the function $f(z)e^{-|z|^2/2}$ belongs to
$L^p(\C, dA)$, where $dA$ is ordinary area measure on the complex
plane. For $f\in F^p$ with $0<p< \infty$ we write
$$\|f\|_p^p=\frac{p}{2\pi}\inc \left|f(z)e^{-\frac{|z|^2}{2}}\right|^p\,dA(z).$$
For $f\in F^\infty$ we define
$$\|f\|_\infty=\sup\left\{|f(z)|e^{-\frac{|z|^2}{2}}: z\in\C\right\}.$$
The spaces $F^p$ are called Fock spaces.

When $p=2$, $F^2$ is a Hilbert space whose inner product is given by
$$\langle f,g\rangle=\inc f(z)\overline{g(z)}\,d\lambda(z),$$
where
$$d\lambda(z)=\frac1\pi e^{-{|z|^2}}\,dA(z)$$
is the standard Gaussian measure on $\C$. Furthermore, $F^2$ is a reproducing
kernel Hilbert space whose kernel function is $K(z,w)=e^{z\overline w}$.
See \cite{Zhu1} for an introduction to the theory of Fock spaces.

Recently in \cite{DZ3} we introduced and studied a two-parameter family of
integral operators on $L^2(\C, d\lambda)$ and on $F^2$. More specifically,
for any $(s,t)\in\C^2$ with $s\not=0$, we define
\begin{equation}\label{operator_def}
T^{(s,t)}f(z)=\inc K^{(s,t)}(z,w)f(w)\,d\lambda(w),
\end{equation}
where
$$K^{(s,t)}(z,w)=\frac1{\sqrt s}\,\exp\left[\frac{tz^2-\overline{tw^2}
+2z\overline w}{2s}\right].$$
Here and throughout the paper, the complex square root is defined as follows:
$$\sqrt z=\sqrt{|z|}\,e^{i\theta/2},\qquad z=|z|e^{i\theta},
\qquad \theta\in(-\pi, \pi].$$
Note that $K^{(1,0)}(z,w)=e^{z\overline{w}}$ is the reproducing kernel
of $F^2$, and $T^{(1,0)}=P$ is the orthogonal projection from
$L^2(\C, d\lambda)$ onto $F^2$.

One of the main results of \cite{DZ3} is that $T^{(s,t)}$ is bounded on $F^2$
if and only if $|s|^2\ge|t|^2+1$. Furthermore, $|s|^2=|t|^2+1$ if and only
if $T^{(s,t)}$ is unitary, and $|s|^2>|t|^2+1$ if and only if $T^{(s,t)}$ is
compact. Thus when $T^{(s,t)}$ is bounded, it is either unitary or compact.
Most of \cite{DZ3} was devoted to the study of $T^{(s,t)}$ when
they are unitary. In particular, it was shown in \cite{DZ3} that,  when 
$|s|^2=|t|^2+1$, the operators $T^{(s,t)}$ on $F^2$ are unitarily 
equivalent to the classical {\it linear canonical 
transforms} on $L^2(\R, dx)$, which are widely used in mathematical 
physics and engineering applications (see \cite{HKOS} for example). 
For this reason we called the two-parameter family of integral operators 
$T^{(s,t)}$ {\it canonical integral operators} on the Fock space.

In this paper, we continue this program and study $T^{(s,t)}$ when they 
are compact. The most important properties of a compact operator 
are its singular values, which will be our main focus here.
The main results of this paper are Theorems A and B below.

\begin{thma}
Let $p>0$ and $(s, t)\in\C\times\C$ with $|s|^2>|t|^2+1$. Then the
compact operator $T^{(s,t)}: F^2\to F^2$ belongs to the Schatten class
$S_p$. Moreover,
\begin{enumerate}
\item[(a)] The decreasing sequence $\left\{\mu_n\right\}_{n=0}^\infty$,
where
$$\mu_n=\left[\frac{2|s|}{|s|^2-|t|^2-1+
\sqrt{(|s|^2-|t|^2-1)^2+4|s|^2}}\right]^{n+(1/2)},$$
is the sequence of all singular values for $T^{(s,t)}$.
\item[(b)] The operator norm and Schatten $p$-norm of $T^{(s,t)}$ 
are given by
$$\|T^{(s,t)}\|=\mu_0,\qquad \|T^{(s,t)}\|_{S^p}=
\mu_0\,\left(1-\mu_0^{2p}\right)^{-(1/p)}.$$ 
\item[(c)] The trace of $T^{(s,t)}$ is given by
$$\tr\left(T^{(s,t)}\right)=
\frac{1}{\sqrt{s}}\sqrt{\frac{s^2}{(s-1)^2+|t|^2}}.$$
\end{enumerate}
\end{thma}

For any fixed $w\in\C$ we write
$$K_w^{(s,t)}(z)=K^{(s,t)}(z,w),\qquad z\in\C.$$
It follows from Lemma 1 of \cite{DZ3} that the function $K_w^{(s,t)}$ 
belongs to $F^2$ if and only if $|s|>|t|$. Therefore, for $|s|>|t|$, the 
integral in (\ref{operator_def}) gives rise to an entire function $T^{(s,t)}f$ 
for every $f\in F^2$. In fact, it is not difficult to see that $T^{(s,t)}f$ is a
well-defined entire function for every $f\in F^\infty$ (and hence for every
$f\in F^p$ with $p>0$).

Since $T^{(s,t)}f$ is a well-defined entire function for every $f\in F^p$ 
when $0<p\le\infty$ and $|s|>|t|$, it is natural for us to consider the 
boundedness and ``compactness'' of $T^{(s,t)}$ on $F^p$. There is a 
minor technical issue we must address, that is, the notion of ``compact 
operators'' on the non-Banach spaces $F^p$ for $0<p<1$ and the 
non-reflexive Banach spaces $F^1$ and $F^\infty$. Thus in the rest of 
the paper, for convenience and simplicity of terminology, we will say that 
a bounded linear operator $T: F^p\to F^p$, where $0<p\le\infty$, is 
compact if $\|Tf_n\|_{p}\to0$ (as $n\to\infty$) for every bounded 
sequence $\{f_n\}\subset F^p$ that converges to $0$ uniformly on 
compact subsets of $\C$. When $1<p<\infty$, this definition of 
compactness is equivalent to the usual notion of compactness on the 
Banach space $F^p$.

\begin{thmb}
Suppose $|s|>|t|$ and $0<p\leq\infty$. Then
\begin{enumerate}
\item[(a)] $T^{(s, t)}$ is bounded on $F^p$ if and only if $|s|^2\ge|t|^2+1$;
\item[(b)] $T^{(s,t)}$ is compact on $F^p$ if and only if $|s|^2>|t|^2+1$.
\end{enumerate}
\end{thmb}

Consequently, the boundedness (and compactness) of $T^{(s,t)}$ on $F^p$
is actually independent of $p$, where $0<p\le\infty$.

\section{Estimates for canonical integral kernels}

For any $w\in\C$ we use $K_w$ to denote the reproducing kernel of $F^2$ 
at $w$. Thus $K_w(z)=K(z,w)=e^{z\overline w}$. The function $k_w=
K_w/\|K_w\|_2$ is called the normalized reproducing kernel of
$F^2$ at $w$. Each $k_w$ is a unit vector in $F^2$; it is also a unit
vector in $F^p$ for any $0<p\le\infty$. See \cite{Zhu1}. We begin with the 
following integral representation for general operators on the Fock space.

\begin{proposition}\label{1}
If $T$ is a bounded linear operator on $F^2$, then there exists a unique
function $H(z,w)$ on $\C\times\C$ with the following properties.
\begin{itemize}
\item[(i)] For fixed $z,w\in\C$ the functions $u\mapsto H(u,w)$
and $u\mapsto\overline{H(z,u)}$ are both in $F^2$.
\item[(ii)] The operator $T$ is represented as
\begin{equation}\label{integral_rep_for_T}
Tf(z)=\inc H(z,w)f(w)\,d\lambda(w),\qquad f\in F^2,\quad z\in\C.
\end{equation}
\end{itemize}
\end{proposition}

\begin{proof}
Note that for any $f\in F^2$ and $z\in\C$ we have
\begin{align*}
Tf(z)&=\langle Tf, K_z\rangle=\langle f, T^*K_z\rangle\\
&=\inc\overline{(T^*K_z)(w)}f(w)\,d\lambda(w)\\
&=\inc H(z,w)f(w)\,d\lambda(w),
\end{align*}
where
\begin{align*}
H(z,w)&=\overline{(T^*K_z)(w)}=\overline{\langle T^*K_z,
K_w\rangle}=\langle K_w, T^*K_z\rangle\\
&=\langle TK_w, K_z\rangle=(TK_w)(z).
\end{align*}
For any fixed $w$ the function $z\mapsto H(z,w)$ is simply
$TK_w$ in $F^2$. Similarly, for any fixed $z\in\C$, the function
$w\mapsto\overline{H(z,w)}$ is simply $T^*K_z$ in $F^2$.

On the other hand, if $T$ admits the integral representation
in (\ref{integral_rep_for_T}) with $H(z,w)$ satisfying the
conditions in (i), then for any $w,z\in\C$ we have
\begin{align*}
(TK_w)(z)&=\inc H(z,u)K_w(u)\,d\lambda(u)\\
&=\overline{\inc\overline{H(z,u)}K(w,u)\,d\lambda(u)}\\
&=\overline{\overline{H(z,w)}}=H(z,w),
\end{align*}
which shows that the integral kernel $H(z,w)$ is {\it uniquely}
determined by $T$.
\end{proof}

It is clear that the integral representation above also holds for
operators on other reproducing kernel Hilbert spaces of analytic
functions. The proof used the boundedness of $T$ to guarantee
that the function $T^*K_z$ is well defined in $F^2$. If
$T: F^2\to F^2$ is an unbounded linear operator whose domain
contains all kernel functions $K_z$, then the integral kernel
$H(z,w)=(TK_w)(z)$ is still well defined, but it is not clear if the
function $u\mapsto\overline{H(z,u)}$ must be in $F^2$, and so it
is not clear if the result is still correct in this case. Nevertheless,
we still have the following result for certain canonical
integral operators on $F^2$ that are not necessarily bounded. Recall
from \cite{DZ3} that $T^{(s,t)}$ is bounded on $F^2$ if and only
if $|s|^2\geq |t|^2+1$.

\begin{proposition}\label{propBT}
Let $(s, t)\in\mathbb{C}\times\mathbb{C}$ with $|s|>|t|$. Then
$$K^{(s,t)}(z, w)=\langle T^{(s,t)}K_w, K_z\rangle
=e^{\frac{|z|^2+|w|^2}{2}}\langle T^{(s,t)}k_w, k_z\rangle$$
for all $z,w\in\C$.
\end{proposition}

\begin{proof}
Recall that for any $w\in\C$ we write
$$K^{(s,t)}_w(z)=K^{(s,t)}(z,w),\qquad z\in\C.$$
It follows from Lemma 1 of \cite{DZ3} that $K^{(s,t)}_w\in F^2$ if and
only if $|s|>|t|$. It is easy to see that
\begin{equation}\label{eq3}
\overline{K^{(s,t)}(z,w)}=\frac{\sqrt{\overline s}}{\overline{\sqrt s}}
\,K^{(\overline s,-t)}(w,z).
\end{equation}
So the function $w\mapsto\overline{K^{(s,t)}(z,w)}$ belongs to
$F^2$ if and only if $|s|>|t|$. By the second half of the proof of
Proposition~\ref{1}, we have for $|s|>|t|$ that
\begin{equation}\label{eqTKw}
\left(T^{(s,t)}K_w\right)(u)=K^{(s,t)}(u,w)=K^{(s,t)}_w(u).
\end{equation}
Consequently, for $|s|>|t|$, we have
$$\langle T^{(s,t)}K_w, K_z\rangle=\langle K^{(s,t)}_w, K_z\rangle
=K^{(s,t)}_w(z)=K^{(s,t)}(z,w)$$
for all $z,w\in\C$.
\end{proof}

In this section we will obtain an upper bound for the two-parameter
integral kernels $K^{(s,t)}(z,w)$, which we will call {\it canonical integral
kernels}. This estimate in turn will enable us to prove several new results
later in the paper about the corresponding canonical integral operators
$T^{(s,t)}$.

\begin{corollary}\label{CorKleq}
Let $(s, t)\in\C\times\C$ with $|s|>|t|$. Then
$$\left|K^{(s,t)}(z,w)\right|\le\inc\left|T^{(s,t)}K_w(v)K_z(v)
\right|\,d\lambda(v)$$
and
$$\left|K^{(s,t)}(z,w)\right|\le\inc\left|{T^{(\overline{s},-t)}
K_z(v)K_w(v)}\right|\,d\lambda(v).$$
\end{corollary}

\begin{proof}
The first inequality follows directly from Proposition~\ref{propBT}.
By (\ref{eq3}), we have
$$\left|K^{(s,t)}(z,w)\right|=\left|K^{(\,\overline{s},-{t})}(w,z)\right|,$$
which then yields the second inequality.
\end{proof}

Recall from Equation (1.18) of \cite{B1} that
\begin{equation}\label{eqIC}
\inc e^{\frac{1}{2}\gamma w^{2}+aw}e^{\frac{1}{2}\overline{\delta w^2}
+\overline{bw}}\,d\lambda(w)
=\frac{1}{\sqrt{1-\gamma\overline{\delta}}}
\exp\left[\frac{\overline{\delta}a^{2}+\gamma \overline{b}^{2}
+2a\overline{b}}{2(1-\gamma\overline{\delta})}\right]
\end{equation}
under the assumption $|\gamma+\delta|^{2}<4$, which implies that
$$\re(1-\gamma\overline\delta)=1-\frac14|\gamma+\delta|^2+\frac14|
\gamma-\delta|^2>0.$$

\begin{lemma}\label{lemiTK}
Let $(s, t)\in\C\times\C$ with $|s|>|t|$. Then for all $z, w\in\C$ we have
\begin{multline*}
\inc\left|T^{(s,t)}K_w(v)K_z(v)\right|\,d\lambda(v)\\
=\frac{2\sqrt{|s|}}{\sqrt{4|s|^2-|t|^2}}
\left|\exp\left(-\frac{\overline{t}}{2s}\overline{w}^2+\frac{s{t}
\left(w+\overline{s}z\right)^{2}}{2|s|^2(4|s|^2-|t|^2)}\right)\right|
\exp\left(\frac{\left|{w}+\overline{s}{z}\right|^2}{4|s|^2-|t|^2}\right)
\end{multline*}
and
\begin{multline*}
\inc\left|{T^{(\overline{s},-t)}K_z(v)K_w(v)}\right|\,d\lambda(v)\\
=\frac{2\sqrt{|s|}}{\sqrt{4|s|^2-|t|^2} }\left|\exp\left(\frac{{t}}{2{s}}{z}^2
-\frac{\overline{s}{t}\left( z+{s}w\right)^{2}}{2|s|^2(4|s|^2-|t|^2)}\right)
\right|\exp\left(\frac{\left|{ z}+{s}{w}\right|^2}{4|s|^2-|t|^2}\right).
\end{multline*}
\end{lemma}

\begin{proof}
It is clear from \eqref{eqTKw} that
\begin{multline*}
\inc \left|T^{(s,t)}K_w(v)K_z(v)\right|d\lambda(v)\\
=\frac{1}{\sqrt{|s|}}\left|\exp\left(-\frac{\overline{t}}{2s}\overline{w}^2
\right)\right|\inc \left|\exp\left(\frac{t}{2s}v^2+\frac{ v\overline{w}}{s}
+v\overline{z}\right)\right|d\lambda(v).
\end{multline*}
Since $|t|<2|s|$, it follows from \eqref{eqIC} that
\begin{align*}
&\inc\left|\exp\left(\frac{t}{4s}v^2+\frac{ v\overline{w}}{2s}+
\frac{v\overline{z}}{2}\right)\right|^2 d\lambda(v)\\
&\quad=\frac{2|s|}{\sqrt{4|s|^2-|t|^2} }\exp \left(\frac{s\overline{t}
\left(\frac{\overline{w}}{2s}+\frac{\overline{z }}{2}\right)^{2}+
t\overline{s}\left(\frac{ w}{2\overline{s}}+\frac{z}{2}\right)^{2}+
\left(\overline{w}+s\overline{z}\right)\left({ w}+\overline{s}{z}\right)}
{4|s|^2-|t|^2}\right)\\
&\quad=\frac{2|s|}{\sqrt{4|s|^2-|t|^2}}\left|\exp\left(\frac{s{t}
\left(w+\overline{s}z\right)^{2}}{2|s|^2(4|s|^2-|t|^2)}\right)\right|
\exp\left(\frac{\left|{w}+\overline{s}{z}\right|^2}{4|s|^2-|t|^2}\right).
\end{align*}
This proves the first equality. Replace $(s,t)$ by $(\overline{s},-t)$ and 
interchange $z$ and $w$ in the first equality. Then we obtain the second 
equality.
\end{proof}

We are now ready to prove the main result of this section, which gives
an useful upper bound for the two-parameter canonical integral kernel.

\begin{theorem}\label{thmKzw}
Let $(s, t)\in\C\times\C$ with $|s|>|t|$. Then
$$\left| K^{(s,t)}(z,w)\right|\leq C\,\exp\left[\frac{(|s|^2+1)
(|z|^2+|w|^2)}{3|s|^2-|t|^2+1}-\frac{(|s|^2-|t|^2-1)
\re({s\overline{z}w})}{|s|^2(3|s|^2-|t|^2+1)}\right]$$
for any $z, w\in\C$, where
\begin{align*}
C=\frac{1}{\sqrt{|s|}}\left[\frac{2{|s|}}{\sqrt{4|s|^2-|t|^2}}
\right]^{\displaystyle2(4|s|^2-|t|^2)/(3|s|^2-|t|^2+1)}
\geq\frac{1}{\sqrt{|s|}}.
\end{align*}
\end{theorem}

\begin{proof}
A direct calculation shows that
$$s{t}\left(w+\overline sz\right)^2-\overline st\left( z+sw\right)^2
=t(|s|^2-1)(\overline sz^{2}-sw^2)$$
and
$$\left|{z}+{s}{w}\right|^2+\left|{w}+\overline{s}{z}\right|^2=
(1+|s|^2)\left(|z|^2+|w|^2\right)+4\re ({s\overline{z}w}).$$
By Corollary \ref{CorKleq} and Lemma \ref{lemiTK}, we have
\begin{align*}
\left|K^{(s,t)}(z,w)\right|^2&\leq\inc\left|T^{(s,t)}K_w(v)K_z(v)\right|
\,d\lambda(v)\inc\left|{T^{(\overline{s},-t)}K_z(v)K_w(v)}\right|\,
d\lambda(v)\\[8pt]
&=\frac{4|s|}{{4|s|^2-|t|^2} }\left|\exp\left(\frac{{t}}{2{s}}{z}^2
-\frac{\overline{t}}{2s}\overline{w}^2+\frac{{t}(|s|^2-1)
(\overline{s}z^{2}-sw^2)}{2|s|^2(4|s|^2-|t|^2)}\right)\right|\\[8pt]
&\quad\cdot\,\exp\left(\frac{(1+|s|^2)\left(|z|^2+|w|^2\right)
+4\re ({s\overline{z}w})}{4|s|^2-|t|^2}\right).
\end{align*}
Since
$$\left|\exp\left(\frac{t}{2s}z^2-\frac{\overline{t}}{2s}\overline{w}^2
\right)\right|=\left|\exp\left(\frac{{t}(\overline{s}{z}^2-s{w}^2)}{2{|s|^2}}
\right)\right|,$$
it follows that
\begin{align}\label{eqTKzw}
\left| K^{(s,t)}(z,w)\right|&\leq\frac{2\sqrt{|s|}}{\sqrt{4|s|^2-|t|^2} }
\left|\exp\left(\frac{t}{2s}z^2-\frac{\overline{t}}{2s}\overline{w}^2\right)
\right|^{(5|s|^2-|t|^2-1)/[2(4|s|^2-|t|^2)]}\nonumber\\[8pt]
&\quad\cdot\,\exp\left(\frac{(|s|^2+1)(|z|^2+|w|^2)+4\re ({s\overline{z}w})}
{2(4|s|^2-|t|^2)}\right).
\end{align}
On the other hand, it is clear from the definition of $K^{(s,t)}(z,w)$ that
\begin{equation}\label{eqnormK}
\left| K^{(s,t)}(z,w)\right|=\frac{1}{\sqrt{|s|}}\left|\exp\left(\frac{t}{2s}z^2
-\frac{\overline{t}}{2s}\overline{w}^2\right)\right|
\exp\left(\frac{\re ({s\overline{z}w})}{|s|^2}\right).
\end{equation}
Combining \eqref{eqTKzw} with \eqref{eqnormK}, we obtain
\begin{align*}
\left|\exp\left(\frac{t}{2s}z^2-\frac{\overline{t}}{2s}\overline{w}^2
\right)\right|&\leq\left[\frac{2{|s|}}{\sqrt{4|s|^2-|t|^2}}
\right]^{2(4|s|^2-|t|^2)/(3|s|^2-|t|^2+1)}\\[8pt]
&\quad\cdot\,\exp\left(\frac{(|s|^2+1)(|z|^2+|w|^2)}{3|s|^2-|t|^2+1}
-\frac{2(2|s|^2-|t|^2)\re({s\overline{z}w})}{|s|^2(3|s|^2-|t|^2+1)}\right).
\end{align*}
Using this and \eqref{eqnormK} again, we deduce that
\begin{align*}
&\left|K^{(s,t)}(z,w)\right|\\[8pt]
&\leq C\,\exp\left(\frac{(|s|^2+1)(|z|^2+|w|^2)}
{3|s|^2-|t|^2+1}-\frac{2(2|s|^2-|t|^2)\re ({s\overline{z}w})}
{|s|^2(3|s|^2-|t|^2+1)}\right)\exp\left(\frac{\re ({s\overline{z}w})}
{|s|^2}\right)\\[8pt]
&=C\,\exp\left(\frac{(|s|^2+1)(|z|^2+|w|^2)}{3|s|^2-|t|^2+1}
-\frac{(|s|^2-|t|^2-1)\re({s\overline{z}w})}{|s|^2(3|s|^2-|t|^2+1)}\right).
\end{align*}
This completes the proof of the theorem.
\end{proof}

\begin{remark}
For $t=0$ we deduce directly from \eqref{eqTKzw} that
$$\left|K^{(s,t)}(z,w)\right|\leq\frac{1}{\sqrt{|s|}}\,
\exp\left[\frac{(|s|^2+1)(|z|^2+|w|^2)+4\re ({s\overline{z}w})}
{8|s|^2}\right],$$
which is more precise than the estimate in Theorem~\ref{thmKzw}.
This inequality can also be proved directly using the elementary inequality
$|a|^2+ |b|^2\geq 2|ab|$.
\end{remark}

Recall that if $T: F^2\to F^2$ is a (possibly unbounded) linear operator
whose domain contains all kernel functions $K_z$, then the function
$$(z,w)\in\C\times\C\mapsto\langle Tk_z, k_w\rangle$$
is called the bivariate Berezin transform of $T$. The restriction of the
bivariate Berezin transform to the diagonal, namely, the function
$$z\in\C\mapsto\langle Tk_z, k_z\rangle$$
is called the Berezin transform of $T$. The following result gives an
upper estimate for the bivariate Berezin transform of $T^{(s,t)}$.

\begin{theorem}\label{thmTKzw}
Let $(s, t)\in\C\times\C$ with $|s|>|t|$. Then
$$\left|\langle T^{(s,t)}k_w, k_z\rangle\right|\leq C\,
\exp\left[-\frac{(|s|^2-|t|^2-1)}{2|s|^2(3|s|^2-|t|^2+1)}
\left[|s|^2(|z|^2+|w|^2)+2\re ({s\overline{z}w})\right]\right]$$
for all $z$ and $w$ in $\C$, and consequently,
$$\left|\langle T^{(s,t)}k_z, k_z\rangle\right|\leq C\,
\exp\left[-\frac{(|s|^2-|t|^2-1)\left[|s|^2+\re ({s})\right]}
{|s|^2(3|s|^2-|t|^2+1)}|z|^2\right]$$
for all $z\in\C$. Furthermore, if $|s|^2\ge |t|^2+1$, then
$$ \left|\langle T^{(s,t)}k_w, k_z\rangle\right|\leq C\,
\exp\left[-\frac{(|s|^2-|t|^2-1)}{2|s|^2(3|s|^2-|t|^2+1)}|z+s w|^2
\right].$$
Here $C$ is the constant from Theorem~\ref{thmKzw}.
\end{theorem}

\begin{proof}
By Proposition~\ref{propBT} and Theorem~\ref{thmKzw}, we have
\begin{align*}
&\left|\langle T^{(s,t)}k_w, k_z\rangle\right|\\[8pt]
&\leq C\,\exp\left[-\frac{|z|^2+|w|^2}{2}\right]
\exp\left[\frac{(|s|^2+1)(|z|^2+|w|^2)}{3|s|^2-|t|^2+1}
-\frac{(|s|^2-|t|^2-1)\re ({s\overline{z}w})}{|s|^2(3|s|^2-|t|^2+1)}
\right]\\[8pt]
&=C\,\exp\left[-\frac{(|s|^2-|t|^2-1)}{2|s|^2(3|s|^2-|t|^2+1)}
\left[|s|^2(|z|^2+|w|^2)+2\re ({s\overline{z}w})\right]\right]
\end{align*}
for any $(s, t)\in\C\times\C$ with $|s|>|t|$, as desired. In particular,
if $z=w$, it is clear that
$$\left|\langle T^{(s,t)}k_z, k_z\rangle\right|\leq C\,\exp
\left[-\frac{(|s|^2-|t|^2-1)}{|s|^2(3|s|^2-|t|^2+1)}
\left[|s|^2+\re ({s})\right]|z|^2\right].$$

If we further assume $|s|^2\ge|t|^2+1$, which is equivalent to the
boundedness of $T^{(s,t)}$ on $F^{2}$, then the additional estimates
clearly follow from the elementary inequality
$$|s|^2(|z|^2+|w|^2)+2\re ({s\overline{z}w})\geq |z+s w|^2 .$$
This completes the proof of the theorem.
\end{proof}

Recall that $T^{(s,t)}: F^2\to F^2$ is compact if and only if
$|s|^2>|t|^2+1$. When a linear operator $T$ on $F^2$ is compact, its
Berezin transform has the property that $\langle Tk_z, k_z\rangle\to0$
as $z\to\infty$, because $k_z\to0$ weakly in
$F^2$ as $z\to\infty$. The estimate on the Berezin transform of $T^{(s,t)}$
in Theorem~\ref{thmTKzw} clearly confirms this and gives more information:
when $T^{(s,t)}$ is compact on $F^2$, its Berezin tranform tends to $0$
very rapidly (i.e. second-order exponentially) as $z\to\infty$. This is a clear
indication why the results in the next section are reasonable.

Also note that, by Proposition~\ref{propBT}, the Berezin transform of
$T^{(s,t)}$ is given by
$$\langle T^{(s,t)}k_z, k_z\rangle=e^{-|z|^2}K^{(s,t)}(z,z)=
\frac1{\sqrt s}\exp\left[\frac{tz^2-\overline{tz^2}+2|z|^2}{2s}
-|z|^2\right].$$
From this alone it is not even clear that the function above, when
$|s|^2>|t|^2+1$, must tend to $0$ as $z\to\infty$. So the estimates
in this section are highly non-trivial.

On the other hand, for the special case when $s$ is real, the explicit
formula in the previous paragraph yields
$$|\langle T^{(s,t)}k_z, k_z\rangle|=\frac1{\sqrt{|s|}}
\exp\left[\left(\frac1s-1\right)|z|^2\right].$$
In this case, if $|s|^2>|t|^2+1$, then we must have $s>1$ or $s<-1$,
which clearly shows that $\langle T^{(s,t)}k_z, k_z\rangle\to0$
second-order exponentially as $z\to\infty$.

Another interesting (and somewhat surprising) case is when $s$ is real
and $s^2=|t|^2+1$. In this case, the operator $T^{(s,t)}$ is unitary
(this includes a lot of classical linear canonical transforms under the
inverse Bargmann transform) but its Berezin transform still tends to $0$ 
second-order exponentially as $z\to\infty$ unless $t=0$ and $s=1$. 
When $t=0$ and $s=1$, $T^{(s,t)}$ is the identity operator on $F^2$ 
and its Berezin transform is the constant function $1$.


\section{Membership of $T^{(s,t)}$ in Schatten classes}

It is a classical result in functional analysis that if $T$ is positive and compact
on a Hilbert space $\mathcal{H}$, then there exists an orthonormal set
$\{e_n\}$ in $\mathcal{H}$ and a non-increasing sequence $\{\lambda_n\}$
of positive numbers such that
$$T(x)=\sum_{n}\lambda_n\langle x,e_n\rangle e_n,\qquad x\in\mathcal{H}.$$
The numbers $\lambda_n$ are uniquely determined by $T$ and are called
the singular values of $T$.

We say that a compact linear operator $T$ on $\mathcal H$ belongs to the
Schatten class $S_p$, $0<p<\infty$, if the singular-value sequence
$\{\lambda_n\}$ of $|T|=(T^*T)^{1/2}$ belongs to $l^p$. Thus when we
talk about singular values for $T$ we simply mean singular values for $|T|$.
If $T\in S_p$ with singular values $\{\lambda_n\}$, we will write
$$\|T\|_{S_p}=\left(\sum_n\lambda_n^p\right)^{1/p}.$$
We refer the reader to \cite{Zhu2} for more information about
Schatten classes.

In this section we will show that each compact canonical integral operator
$T^{(s,t)}$ on $F^2$ actually belongs to the Schatten class $S_p$ for all 
$p>0$. We will be able to accomplish this by using estimates of the canonical
integral kernels but without actually knowing the singular values. In the
next section we will then compute the singular values for $T^{(s,t)}$.
The results of this section will help us determine that the singular values
we find in the next section are {\it all} the singular values.

When the integral operator $T^{(s,t)}$ is bounded on $F^2$, it is easy to 
see that its adjoint is also an integral operator and is given by
$$\left(T^{(s,t)}\right)^*f(z)=\inc f(w)\overline{K^{(s,t)}(w,z)}\,
d\lambda(w).$$
By (\ref{eq3}), we have
\begin{align}\label{eqTstar}
\left(T^{(s,t)}\right)^*f(z)=\frac{\sqrt{\overline{s}}}{\overline{\sqrt{s}}}
\inc f(w)K^{(\overline{s},-t)}(z,w)\,d\lambda(w)
=\frac{\sqrt{\overline{s}}}{\overline{\sqrt{s}}}\,T^{(\overline{s},-t)}f(z).
\end{align}
We will follow the standard practice of writing
$$\left|T^{(s,t)}\right|^2=\left(T^{(s,t)}\right)^*T^{(s,t)}$$
on the Fock space $F^2$.

\begin{lemma}\label{lemma|T|}
Let $(s, t)\in\C\times\C$ with $|s|^2\ge|t|^2+1$. Then for any
$f\in F^{2}$ we have
\begin{multline*}
\left|T^{(s,t)}\right|^2f(z)=\\
\frac{1}{\sqrt{|s|^2-|t|^2}}\inc f(u)\exp\left[\frac{ z\overline{u}}
{{|s|^2-|t|^2}}+\frac{|t|^2+1-|s|^2}{|s|^2-|t|^2}
\left(\frac{t}{2\overline{s}}\,z^2+\frac{\bar t}{2s}
\,\overline u^2\right)\right] d\lambda(u).
\end{multline*}
\end{lemma}

\begin{proof}
It follows from \eqref{eqTstar}, Lemma 2 of \cite{DZ3},
and Fubini's theorem that
\begin{align*}
\left|T^{(s,t)}\right|^2f(z)
&=\frac{\sqrt{\overline{s}}}{\overline{\sqrt{s}}}
\inc T^{(s,t)}f(w) K^{(\overline{s},-t)}(z,w)\,d\lambda(w)\\
&=\frac{\sqrt{\overline{s}}}{\overline{\sqrt{s}}}\inc f(u)\,d\lambda(u)
\inc{K^{(\overline{s},-t)}(z,w)}{K^{({s},t)}(w,u)}\,d\lambda(w)\\
&=c\inc f(u)K^{(|s|^2-|t|^2,0)}(z,u)\exp\left[
\frac{|t|^2+1-|s|^2}{|s|^2-|t|^2}\left(\frac{t}{2\overline{s}}\,z^2
+\frac{\bar t}{2s}
\,\overline u^2\right)\right] d\lambda(u),
\end{align*}
where
$$c=\frac{\sqrt{\overline{s}}}{\overline{\sqrt{s}}}
\frac{\sqrt{|s|^2-|t|^2}}{\sqrt{\overline{s}}\sqrt{s}}
\sqrt{\frac{|s|^2}{|s|^2-|t|^2}}=1.$$
This proves the desired result.
\end{proof}

Recall that for any bounded linear operator $T$ on $F^2$ the function
$\widetilde T(z)=\langle Tk_z, k_z\rangle$ is called the Berezin transform
of $T$. The following result is critical for the main result of this section.

\begin{lemma}\label{lemsp}
Let $p>0$ and $(s, t)\in\C\times\C$ with $|s|^2>|t|^2+1$. Then
\begin{equation}\label{eqsp}
\left\|\widetilde{\left|T^{(s,t)}\right|^2}(z)
\right\|_{L^{{p}/{2}}(\mathbb{C},dA)}=
\left[\frac{2\pi|s|}{p(|s|^2-|t|^2-1)(|s|^2-|t|^2)^\frac{p-2}{4}}
\right]^{\frac{2}{p}}.
\end{equation}
\end{lemma}

\begin{proof}
It is clear from \eqref{eqTKw} that
$$\widetilde{\left|T^{(s,t)}\right|^2}(z)=\left \langle\left(T^{(s,t)}
\right)^*T^{(s,t)}k_z,k_z\right\rangle=e^{-|z|^2}\left\|T^{(s,t)}K_z
\right\|_2^2=e^{-|z|^2}\left\|K_z^{(s,t)}\right\|_2^2.$$
Then a calculation using Corollary 3 of \cite{DZ3} shows that
\begin{align*}
&\left\|\widetilde{\left|T^{(s,t)}\right|^2}(z)\right\|_{L^{{p}/{2}}
(\mathbb{C},dA)}^{{p}/{2}}
=\inc \left\|K_z^{(s,t)} \right\|_2^p e^{-\frac{p}{2}|z|^2}dA(z)\\[8pt]
&=\frac{1}{(|s|^2-|t|^2)^{p/4}}\inc\left|\exp\left[\frac{tp(|t|^2+1-|s|^2)}
{2\overline{s}(|s|^2-|t|^2)}{z^2}\right]\right|\!\exp
\left[\frac{-p(|s|^2-|t|^2-1)|z|^2}{2(|s|^2-|t|^2)}\right]dA(z).
\end{align*}
With the change of variables
$$u=\sqrt{\frac{p(|s|^2-|t|^2-1)}{2(|s|^2-|t|^2)}}\,z$$
and with the help of \eqref{eqIC}, we obtain
\begin{align*}
&\inc
\left|\exp\left[\frac{tp(|t|^2+1-|s|^2)}{2\overline{s}(|s|^2-|t|^2)}{z^2}
\right]\right| \! \exp \left[\frac{-p(|s|^2-|t|^2-1)|z|^2}{2(|s|^2-|t|^2)}\right]dA(z)\\[8pt]
&\qquad=\frac{2\pi(|s|^2-|t|^2)}{p(|s|^2-|t|^2-1)}
\inc\left|e^{-tu^2/(2\overline s)}\right|^2\,d\lambda(u)\\[8pt]
&\qquad=\frac{2\pi(|s|^2-|t|^2)}{p(|s|^2-|t|^2-1)}\frac{|s|}{\sqrt{|s|^2-|t|^2}},
\end{align*}
from which the desired result clearly follows.
\end{proof}

We are now ready to prove the main result of this section.

\begin{theorem}\label{thmSp}
If $|s|^2>|t|^2+1$, then the compact operator $T^{(s,t)}: F^2\to F^2$
belongs to the Schatten class $S_p$ for any $p>0$. Moreover,
$$\|T^{(s,t)}\|_{S_p}^p\le\frac{2|s|}{p(|s|^2-|t|^2-1)
(|s|^2-|t|^2)^{(p-2)/4}},\qquad 0< p\leq2,$$
and
$$\|T^{(s,t)}\|_{S_p}^p\ge\frac{2|s|}{p(|s|^2-|t|^2-1)
(|s|^2-|t|^2)^{(p-2)/4}},\qquad p\geq2.$$
\end{theorem}

\begin{proof}
Fix any $p>0$ and write
$$\left|T^{(s,t)}\right|=\sqrt{\left(T^{(s,t)}\right)^*T^{(s,t)}},\qquad
S=\sqrt{\left|T^{(s,t)}\right|^{{p}}}=\left|T^{(s,t)}\right|^{p/2}.$$
By Theorem 1.26 of \cite{Zhu2}, $T^{(s,t)}$ belongs to $S_p$
if and only if $\left|T^{(s,t)}\right|^{p}\in{S_1}$. Furthermore, for
any orthonormal basis $\{e_n\}$ of $F^2$, we have
$$\|T^{(s,t)}\|_{S_p}^p=\sum_{n=1}^\infty\left \langle \left|T^{(s,t)}
\right|^{p} e_n,e_n\right\rangle=\sum_{n=1}^\infty\inc\left|
\left\langle S e_n,K_z\right\rangle\right|^2\,d\lambda(z).$$
By Fubini's theorem and Parseval's identity, we obtain
\begin{align*}
\|T^{(s,t)}\|_{S_p}^p&=\inc \sum_{n=1}^\infty
\left|\left \langle S K_z, e_n\right\rangle\right|^2d\lambda(z)
=\inc\left\|SK_z\right\|_2^2d\lambda(z)\\
&=\inc\left\langle\left|T^{(s,t)}\right|^{p} k_z,k_z\right\rangle
e^{|z|^2}d\lambda(z).
\end{align*}
If $0< p\leq2$, it follows from Lemma 3.4 of \cite{Zhu1} that
$$\|T^{(s,t)}\|_{S_p}^p\leq\inc\left\langle \left|T^{(s,t)}
\right|^2 k_z,k_z\right\rangle^{p/2} e^{|z|^2}d\lambda(z)
=\frac{1}{\pi}\left\|\widetilde{\left|T^{(s,t)}\right|^2}(z)
\right\|_{L^{{p}/{2}}(\mathbb{C},dA)}^{{p}/{2}}.$$
Combining this with Lemma~\ref{lemsp}, we see that $T^{(s,t)}$
belongs to $S_p$ for $0< p\leq2$. Since $S_p\subset S_q$ for
$0<p\le q<\infty$, we must have $T^{(s,t)}\in S_p$ for $p>2$
as well.

If $p\geq2$, it follows from Lemma 3.4 of \cite{Zhu1} again that
\begin{align*}
\|T^{(s,t)}\|_{S_p}^p\geq\inc\left\langle\left|T^{(s,t)}
\right|^2 k_z,k_z\right\rangle^{p/2}e^{|z|^2}d\lambda(z)
=\frac{1}{\pi}\left\|\widetilde{\left|T^{(s,t)}\right|^2}(z)
\right\|_{L^{{p}/{2}}(\C,dA)}^{{p}/{2}}.
\end{align*}
This completes the proof of the theorem.
\end{proof}

As a direct consequence of Theorem \ref{thmSp}, we obtain
the following result which was also obtained in \cite{DZ3}. The
proof here is different.

\begin{corollary}\label{corS2}
Let $(s, t)\in\C\times\C$ with $|s|^2-|t|^2>1$. Then the operator
$T^{(s,t)}: F^2\to F^2$ is Hilbert-Schmidt with
$$\|T^{(s,t)}\|_{S_2}^2=\frac{|s|}{{|s|^2-|t|^2-1}}.$$
\end{corollary}

Finally in this section we obtain the following trace formula for $T^{(s,t)}$.

\begin{theorem}
If $|s|^2>|t|^2+1$, then
$$\tr\left(T^{(s,t)}\right)=
\frac{1}{\sqrt{s}}\sqrt{\frac{s^2}{(s-1)^2+|t|^2}}.$$
\end{theorem}

\begin{proof}
By Proposition 3.3 of \cite{Zhu1},
$$\tr\left(T^{(s,t)}\right)=\frac1{\pi}\inc\langle Tk_z, k_z\rangle\,dA(z).$$
It then follows from Proposition \ref{propBT} that
\begin{align*}
\tr\left(T^{(s,t)}\right)
&=\frac1{\pi\,\sqrt s}\inc\exp\left[\frac{tz^2}{2s}-\frac{\overline{tz^2}}{2s}
+\frac{1-s}s|z|^2\right]\,dA(z).
\end{align*}
If we write $z=x+iy$, then
\begin{multline}\label{eqItr}
\tr\left(T^{(s,t)}\right)\\
=\frac{1}{\pi\sqrt{s}}\inr\inr\exp\left[\frac{t-\overline{t}-2s+2}{2s}x^2+
\frac{-t+\overline{t}-2s+2}{2s}y^2+\frac{t+\overline{t}}{s}ixy\right]\,dxdy.
\end{multline}
Consider the matrix
$$A=\begin{bmatrix}\displaystyle -\frac{t-\overline{t}-2s+2}{2s} &
\displaystyle -\frac{t+\overline{t}}{2s}i\\[10pt]
\displaystyle -\frac{t+\overline{t}}{2s}i &
\displaystyle -\frac{-t+\overline{t}-2s+2}{2s}
\end{bmatrix}$$
and its real part
$$\re(A)=\begin{bmatrix}
\displaystyle\frac{|s|^2-\im t \im s-\re s}{|s|^2} &
\displaystyle -\frac{\re t\im s}{|s|^2}\\[10pt]
\displaystyle -\frac{\re t\im s}{|s|^2} &
\displaystyle\frac{|s|^2+\im t \im s-\re s}{|s|^2}
\end{bmatrix}.$$
Elementary calculations show that
\begin{align*}
\frac{|s|^2-(\im t \im s+\re s)}{|s|^2}
&>\frac{|s|^2-|\re s|-\sqrt{|s|^2-1} \, |\im s|}{|s|^2}\\
&=\frac{(|\re s|-1)^2}{|s|^2-|\re s|+\sqrt{|s|^2-1} |\im s|}\geq0,
\end{align*}
and
\begin{align*}
\det\re(A)&=\frac{(|s|^2-\re s)^2-(\im s)^2|t|^2}{|s|^4}\\
&>\frac{|s|^4+(\re s)^2-2\re s |s|^2-(\im s)^2(|s|^2-1)}{|s|^4}\\
&=\frac{(\re s-1)^2}{|s|^2}\geq0.
\end{align*}
So the matrix $\re (A)$ ia positive definite, and hence the integral
in \eqref{eqItr} is absolutely convergent. Since
$$\det\, A= \frac{(s-1)^2+|t|^2}{s^2},$$
the desired result then follows from (1.1a) of \cite{B1}.
\end{proof}

\section{Singular values of canonical integral operators}

In this section, we will compute the singular values of the operator
$T^{(s,t)}$ when it is compact. We begin with the following
identification of a special eigenvalue for the operator
$\left|T^{(s,t)}\right|^2$.

\begin{lemma}\label{existenceofgamma}
Let $(s, t)\in\C\times\C$ with $|s|^2>|t|^2+1$. Then
\begin{itemize}
\item[(a)] The equation
\begin{equation}\label{choiceofgamma}
\overline{s t}\gamma^2+(|s|^2+|t|^2+1)\gamma+st=0
\end{equation}
has a unique solution $\gamma$ in the unit disk $\D$.
\item[(b)] For this $\gamma\in\D$ the number
$$\lambda_0=\frac{|s|}{|s|^2+\gamma\overline{st} }$$
is positive.
\item[(c)] The number $\lambda_0$ above is an eigenvalue of
$\left|T^{(s,t)}\right|^2$ and the function $e^{\gamma z^2/2}$ is
a corresponding eigenfunction.
\end{itemize}
\end{lemma}

\begin{proof}
If $t=0$, it is clear that the equation in (\ref{choiceofgamma})
has a unique solution $\gamma=0\in\D$ and $\lambda_0=1$.

If $t\neq0$, then by the quadratic formula, the solutions of
(\ref{choiceofgamma}) are given by
$$\gamma_{\pm}=\frac{-(|s|^2+|t|^2+1)\pm
\sqrt{(|s|^2+|t|^2+1)^2-4|s|^2|t|^2}}{2\overline{st}}.$$
Note that
$$(|s|^2+|t|^2+1)^2-4|s|^2|t|^2=(|s|^2-|t|^2-1)^2+4|s|^2>0.$$
It is easy to check that the solution
$$\gamma_+=\frac{-(|s|^2+|t|^2+1)+\sqrt{(|s|^2+|t|^2+1)^2
-4|s|^2|t|^2}}{2\overline{st}}$$
satisfies
$|\gamma_+|<1$ and
$$|s|^2+\gamma_+\overline{st}=\frac{(|s|^2-|t|^2-1)+
\sqrt{(|s|^2-|t|^2-1)^2+4|s|^2}}{2}>|s|.$$
It is also easy to check that the other solution
$$\gamma_-=\frac{-(|s|^2+|t|^2+1)-\sqrt{(|s|^2+|t|^2+1)^2
-4|s|^2|t|^2}}{2\overline{st}}$$
satisfies $|\gamma_-|>1$. Therefore, the equation in
(\ref{choiceofgamma}) has a unique solution $\gamma=\gamma_+$
in $\D$ with the corresponding $\lambda_0\in(0,1)$.

We now consider the function $f(z)=e^{\gamma z^2/2}\in F^2$.
From \eqref{choiceofgamma} we easily check that
\begin{equation}\label{eqsgamma}
\left(|s|^2+\gamma{\overline{st}}\right)^2
={|s|^2(|s|^2-|t|^2)-\gamma{\overline{st}}(|t|^2+1-|s|^2)}.
\end{equation}
Combining this with Lemma~\ref{lemma|T|} and \eqref{eqIC},
we conclude that
\begin{align*}
\left|T^{(s,t)}\right|^2 f(z)
 &=\frac{1}{\sqrt{|s|^2-|t|^2}}\exp\left[\frac{t(|t|^2+1-|s|^2)}{2\overline{s}(|s|^2-|t|^2)}\,z^2\right]\\[8pt]
&\quad\cdot\,\inc\exp\left[\frac{\gamma u^2}{2}+\frac{z\overline{u}}{{|s|^2-|t|^2}}+
\frac{\bar t(|t|^2+1-|s|^2)}{2s(|s|^2-|t|^2)}\,\overline u^2\right] d\lambda(u)\\
&=\lambda_0\,\exp\left[\frac{st(|t|^2+1-|s|^2)}
{2|s|^2(|s|^2-|t|^2)}\,z^2+\frac{|s|^2 \gamma}{2(|s|^2-|t|^2)\left(|s|^2
+\gamma{\overline{st}}\right)^2}\,z^{2}\right].
\end{align*}
Using \eqref{eqsgamma} again, we have
\begin{align*}
&\frac{st(|t|^2+1-|s|^2)}{|s|^2(|s|^2-|t|^2)}+\frac{|s|^2\gamma}
{(|s|^2-|t|^2)\left(|s|^2+\gamma{\overline{st}}\right)^2}\\
&\qquad=\frac{(|t|^2+1-|s|^2)\left[st(|s|^2-|t|^2)-(|t|^2+1-|s|^2)|t|^2
\gamma\right]+|s|^2 \gamma}{(|s|^2-|t|^2)\left(|s|^2+
\gamma{\overline{st}}\right)^2}\\
&\qquad=\frac{st(|t|^2+1-|s|^2)+\left(|t|^4+2|t|^2+1-|s|^2|t|^2\right)\gamma}{\left(|s|^2+\gamma{\overline{st}}\right)^2}.
\end{align*}
A calculation using \eqref{choiceofgamma} shows that
\begin{multline*}
st(|t|^2+1-|s|^2)+(|t|^4+2|t|^2+1-|s|^2|t|^2)\gamma\\
=\gamma\left[|s|^2(|s|^2-|t|^2)-\gamma{\overline{st}}
(|t|^2+1-|s|^2)\right].
\end{multline*}
From \eqref{eqsgamma} we then deduce that
$$\left|T^{(s,t)}\right|^2 f(z)=\lambda_0\,f(z).$$
This completes the proof of the lemma.
\end{proof}

In order to obtain other eigenvalues of the operator $|T^{(s,t)}|^{2}$
(when it is compact), we will need the following critical result.

\begin{lemma}\label{lempolynom}
Let $(s, t)\in\C\times\C$ with $|s|^2>|t|^2+1$ and let $\gamma$ be
the unique number in $\D$ from Lemma~\ref{existenceofgamma}.
Then, for any polynomial $P_n(x)$ of degree $n$ and any complex
numbers $\nu$ and $b$ satisfying $\re\nu>0$ and
$|2\nu b^2-\gamma|<1$, we have
\begin{multline*}
\left|T^{(s,t)}\right|^2\left[e^{\frac{\gamma}{2}z^2}
\inr  P_n(x)\exp\left[-\nu\left(x-bz\right)^2\right]dx\right]\\[8pt]
=Ce^{\frac{\gamma}{2}z^2}\inr P_n(x)
\exp\left[-\frac{(|s|^2+\gamma{\overline{st}})^{2}\nu
\left(x-\frac{b|s|^2}{(|s|^2+\gamma{\overline{st}})^2}z\right)^2}
{(|s|^2+\gamma{\overline{st}})^2+2\nu b^2{\overline{st}}
(|t|^2+1-|s|^2)}\right]dx,
\end{multline*}
where
$$C=\frac{|s|}{\sqrt{(|s|^2+\gamma
\overline{st})^2+2\nu b^2{\overline{st}}(|t|^2+1-|s|^2)}}$$
with $\re C>0$.
\end{lemma}

\begin{proof}
Note that $\re\nu>0$ and $\gamma\in\D$. It is clear from Lemma 7
of \cite{DZ3} that the function
$$z\mapsto e^{\frac{\gamma}{2}z^2}\inr P_n(x)
\exp\left[-\nu\left(x-bz\right)^2\right]\,dx$$
is in $F^2$. Since $|2\nu b^2-\gamma|<1$ and $|t|<|s|$,
it follows from \eqref{eqsgamma} that
\begin{multline}\label{eqrenu}
\re\left[(|s|^2+\gamma{\overline{st}})^2+2\nu b^2{\overline{st}}
(|t|^2+1-|s|^2)\right]\\
=|s|^2(|s|^2-|t|^2)-(|s|^2-|t|^2-1)\re\left[{\overline{st}}
(2\nu b^2-\gamma)\right]\geq |s|^2,
\end{multline}
which implies that $\re C>0$. Moreover, a calculation using \eqref{eqsgamma}
and \eqref{eqIC} shows that
$$1-(\gamma- 2\nu b^2)
\frac{\bar t(|t|^2+1-|s|^2)}{s(|s|^2-|t|^2)}
=\frac{(|s|^2+\gamma{\overline{st}})^2+2\nu
b^2{\overline{st}}(|t|^2+1-|s|^2)}{|s|^2(|s|^2-|t|^2)}$$
and
\begin{multline*}
\inc\exp\left[\frac{(\gamma- 2\nu b^2) w^2 }{2}+2\nu bxw+
\frac{\bar t(|t|^2+1-|s|^2)}{2s(|s|^2-|t|^2)}\overline{w}^2+
\frac{ z\overline{w}}{|s|^2-|t|^2}\right]\,d\lambda(w)\\[8pt]
=C\sqrt{|s|^2-|t|^2}\exp\left[\frac{2\overline{st}(|t|^2+1-|s|^2)
\nu^2 b^2x^2+2|s|^2\nu b x z+\frac{|s|^2(\gamma- 2\nu b^2)}
{2(|s|^2-|t|^2)}z^2}{(|s|^2+\gamma{\overline{st}})^2+2\nu
b^2{\overline{st}}(|t|^2+1-|s|^2)}\right].
\end{multline*}
Then we use Lemma~\ref{lemma|T|}, Fubini's theorem, and the above
equation to obtain
\begin{align}\label{eq|T|P}
&\left|T^{(s,t)}\right|^2\left[e^{\frac{\gamma}{2}z^2}
\inr P_n(x)\exp\left[-\nu\left(x-bz\right)^2\right]\,dx\right]\nonumber\\
    &=\frac{1}{\sqrt{|s|^2-|t|^2}}\exp\left[\frac{t(|t|^2+1-|s|^2)}{2\overline{s}(|s|^2-|t|^2)}\,z^2\right]\nonumber\\
   & \qquad \cdot \inc e^{\frac{\gamma}{2}w^2+\frac{z\overline{w}}{{|s|^2-|t|^2}}+
\frac{\bar t(|t|^2+1-|s|^2)}{2s(|s|^2-|t|^2)}\,\overline w^2}\,d\lambda(w)\inr P_n(x)e^{-\nu\left(x-bw\right)^2}\,dx\nonumber\\
&= Ce^{\frac{\sigma}{2}z^2}\inr P_n(x)\exp\left[-\frac{(|s|^2+
\gamma{\overline{st}})^2\nu x^2-2|s|^2\nu b x z}{(|s|^2+
\gamma{\overline{st}})^2+2\nu b^2{\overline{st}}
(|t|^2+1-|s|^2)}\right]\,dx,
\end{align}
where
$$\sigma=\frac{st(|t|^2+1-|s|^2)}{|s|^2(|s|^2-|t|^2)}+
\frac{|s|^2(\gamma- 2\nu b^2)}{(|s|^2-|t|^2)\left[(|s|^2+\gamma
{\overline{st}})^2+2\nu b^2{\overline{st}}(|t|^2+1-|s|^2)\right]}.$$

Recall from the proof of Lemma~\ref{existenceofgamma} that
$$\frac{st(|t|^2+1-|s|^2)}{|s|^2(|s|^2-|t|^2)}
 +\frac{|s|^2 \gamma }{(|s|^2-|t|^2)\left(|s|^2+
 \gamma{\overline{st}}\right)^2}={\gamma}.$$
By \eqref{eqsgamma}, we have
\begin{align*}
(|s|^2-|t|^2)(\gamma-\sigma)&=\frac{|s|^2(2\nu b^2-\gamma)}{(|s|^2+\gamma{\overline{st}})^2+2\nu b^2{\overline{st}}(|t|^2+1-|s|^2)}+\frac{|s|^2 \gamma }{(|s|^2+\gamma{\overline{st}})^2}\\
&=\frac{2\nu b^2|s|^2\left[(|s|^2+\gamma{\overline{st}})^2+
\gamma{\overline{st}}(|t|^2+1-|s|^2)\right] }{\left[(|s|^2+
\gamma{\overline{st}})^2+2\nu b^2{\overline{st}}(|t|^2+1-|s|^2)
\right](|s|^2+\gamma{\overline{st}})^2}\\[8pt]
&=\frac{2\nu b^2|s|^4(|s|^2-|t|^2) }{\left[(|s|^2+
\gamma{\overline{st}})^2+2\nu b^2{\overline{st}}
(|t|^2+1-|s|^2)\right](|s|^2+\gamma{\overline{st}})^2}.
\end{align*}
The desired result then follows from \eqref{eq|T|P}.
\end{proof}

We are now ready to determine all eigenvalues for
$\left|T^{(s,t)}\right|^2$ (hence all singular values for
$T^{(s,t)}$ when it is compact) and the corresponding eigenfunctions.
This is done in two steps. The next theorem produces a sequence of
eigenvalues and Theorem~\ref{16} shows that there are no others.

\begin{theorem}\label{thm|T|2}
Let $(s, t)\in\C\times\C$ with $|s|^2>|t|^2+1$ and let $\gamma$ be the
number in $\D$ from Lemma~\ref{existenceofgamma}. Then for each
nonnegative integer $n$ the positive number
$$\lambda_n=\left(\frac{|s|}{|s|^2+
\gamma{\overline{st}}}\right)^{2n+1}$$
is an eigenvalue of $\left|T^{(s,t)}\right|^2$ and the function
$Q_n(z) e^{\gamma z^2/2}$ is a corresponding eigenvector, where
$$Q_n(z)=\inr  H_n(x)\exp\left[-\nu\left(x-bz\right)^2\right]\,dx$$
is a polynomial of degree $n$. Here $H_n(x)$ is the classical Hermite
polynomial of degree $n$,
$$\nu=\frac {(|s|^2+\gamma{\overline{st}})^{2}\left[(|s|^2
+\gamma{\overline{st}})^2+(\gamma+1){\overline{st}}
(|t|^2+1-|s|^2)/2\right]-|s|^4}{(|s|^2+
\gamma{\overline{st}})^{4}-|s|^4},$$
and $b=\sqrt{\gamma+1}/(2 \sqrt{\nu})$.
\end{theorem}

\begin{proof}
From \eqref{eqsgamma} we first observe that
\begin{equation}\label{eqgammaminus}
    \left(|s|^2+\gamma{\overline{st}}\right)^2-|s|^2
=(|s|^2+\gamma{\overline{st}})(|s|^2-|t|^2-1)>0.
\end{equation}
Since $2\nu b^2=(\gamma+1)/2$, we have
$$\nu=\frac {(|s|^2+\gamma{\overline{st}})^{2}\left[(|s|^2+
\gamma{\overline{st}})^2+2\nu b^2{\overline{st}}
(|t|^2+1-|s|^2)\right]-|s|^4}{(|s|^2
+\gamma{\overline{st}})^{4}-|s|^4}.$$
It follows that
$$|2\nu b^2-\gamma|=\left|\frac{1-\gamma}2\right|
\le\frac{1+|\gamma|}2<1,$$
and by \eqref{eqrenu} and \eqref{eqgammaminus},
\begin{equation}\label{eqIRenu}
\re\nu\geq\frac {|s|^2}{(|s|^2+\gamma{\overline{st}})^{2}+|s|^2}>0.
\end{equation}
By Corollary 9 of \cite{DZ3}, $Q_n(z)$ is a polynomial of degree $n$.

According to Lemma~\ref{lempolynom}, it remains for us to prove that
\begin{align}\label{eqintH}
\inr H_n(x)&\exp\left[-\frac{(|s|^2+\gamma{\overline{st}})^{2}\nu
\left(x-\frac{b|s|^2}{\left(|s|^2+\gamma{\overline{st}}\right)^2}z
\right)^2}{(|s|^2+\gamma{\overline{st}})^2+2\nu b^2{\overline{st}}
(|t|^2+1-|s|^2)} \right]dx\nonumber\\[8pt]
&=\frac{\lambda_n}{C}\inr H_n(x)\exp\left[-\nu(x-bz)^2\right]dx,
\end{align}
where $C$ is the constant from Lemma~\ref{lempolynom}.

Let us write
$$a=\frac{b|s|^2}{\left(|s|^2+\gamma{\overline{st}}\right)^2},\qquad
\mu=\frac{\nu(|s|^2+\gamma{\overline{st}})^{2}}{(|s|^2+
\gamma{\overline{st}})^2+2\nu b^2{\overline{st}}(|t|^2+1-|s|^2)}.$$
It is clear from \eqref{eqgammaminus} that $a\neq 0$ and $a^{k}\neq b^{k}$
for $1\le k\leq n$. Since $\gamma\overline{st}$ is real, we have
$$\im\,(2\nu b^2\overline{st})=\im\left(\frac{(\gamma+1)
\overline{st}}{2}\right)=\frac{\im\,(\overline{st})}{2}.$$
Combining this with \eqref{eqrenu} and \eqref{eqIRenu}, we obtain
\begin{align*}
&\frac{\left|(|s|^2+\gamma{\overline{st}})^2+
2\nu b^2{\overline{st}}(|t|^2+1-|s|^2)\right|^2\re \mu }
{(|s|^2+\gamma{\overline{st}})^{2}}\\
&=\re\nu\re\left[(|s|^2+\gamma{\overline{st}})^2+
2{\nu b^2{\overline{st}}}(|t|^2+1-|s|^2)\right]+(|t|^2+1-|s|^2)
\im\nu\im\,(\overline{st})/2 \\
&\geq\frac{|s|^4}{(|s|^2+\gamma{\overline{st}})^{2}+|s|^2}
+\frac{(|s|^2+\gamma{\overline{st}})^2(|t|^2+1-|s|^2)^2
[\im\,(\overline{st})]^2}{4\left[(|s|^2+\gamma{\overline{st}})^{4}
-|s|^4\right]}>0,
\end{align*}
which clearly implies that $\re \mu >0$. Consequently,
$$\frac{\sqrt{\nu}}{\sqrt{\mu}}\left(\frac{a}{b}\right)^n=\frac{|s|^{2n}\sqrt{(|s|^2+
\gamma{\overline{st}})^2+2\nu b^2{\overline{st}}(|t|^2+1-|s|^2)}}{(|s|^2+\gamma{\overline{st}})^{2n+1}}=\frac{\lambda_n}{C}.
$$
Moreover, combining \eqref{eqgammaminus} with \eqref{eqrenu}, we have
\begin{align*}
\nu b^2-\mu a^2=\nu b^2\frac{(|s|^2+\gamma{\overline{st}})^{2}
\left[(|s|^2+\gamma{\overline{st}})^2+2\nu b^2{\overline{st}}
(|t|^2+1-|s|^2)\right]-|s|^4}{(|s|^2+\gamma{\overline{st}})^{2}
\left[(|s|^2+\gamma{\overline{st}})^2+2\nu b^2{\overline{st}}
(|t|^2+1-|s|^2)\right]}\neq0,
\end{align*}
and then
\begin{align*}
\frac{(b^2-a^2)\mu\nu}{\nu b^2-\mu a^2}
&=\frac{\nu b^2}{\nu b^2-\mu a^2}\left(1-\frac{a^2}{b^2}\right)\mu\\
&=\frac{(|s|^2+\gamma{\overline{st}})^{4}\nu-|s|^4\nu}
{(|s|^2+\gamma{\overline{st}})^{2}\left[(|s|^2+
\gamma{\overline{st}})^2+2\nu b^2{\overline{st}}
(|t|^2+1-|s|^2)\right]-|s|^4}=1.
\end{align*}
It follows from Theorem E of \cite{DZ3} that the $n$th Hermite
polynomial $H_n(x)$ is a solution of the integral equation \eqref{eqintH}.
This completes the proof of the theorem.
\end{proof}

\begin{theorem}\label{16}
Let $(s, t)\in\C\times\C$ with $|s|^2>|t|^2+1$ and let $\gamma\in\D$
be the number from Lemma~\ref{existenceofgamma}. Then the
decreasing sequence $\left\{\mu_n\right\}_{n=0}^\infty$, where
$$\mu_n=\left(\frac{|s|}{|s|^2+
\gamma{\overline{st}}}\right)^{n+(1/2)},$$
is the sequence of all singular values for $T^{(s,t)}$. Consequently,
$$\|T^{(s,t)}\|_{S_p}^p=\frac{\left[|s|
(|s|^2+\gamma{\overline{st}})\right]^{p/2}}
{(|s|^2+\gamma{\overline{st}})^p-|s|^p }.$$
\end{theorem}

\begin{proof}
If $\lambda_n$ is the eigenvalue of $\left|T^{(s,t)}\right|^2$ from
Theorem~\ref{thm|T|2}, then $\mu_n=\sqrt{\lambda_n}$ is an eigenvalue
of $\left|T^{(s,t)}\right|$. Thus by definition, each $\mu_n$ is a
singular value of $T^{(s,t)}$.

There might be additional singular values, and some of the eigenvalues
$\lambda_n$ might have high multiplicities, thus we have the following
inequality:
$$\|T^{(s,t)}\|_{S_p}^p\geq\sum_{n=0}^{\infty}\mu_n^p
=\frac{\left[|s|(|s|^2+\gamma{\overline{st}})
\right]^{p/2}}{(|s|^2+\gamma{\overline{st}})^p-|s|^p }.$$
In particular, if $p=2$, then by \eqref{eqgammaminus} we have
$$\|T^{(s,t)}\|_{S_2}^2\geq \sum_{n=0}^{\infty}\mu_n^2=
\frac{|s|}{|s|^2-|t|^2-1 }.$$
However, Corollary \ref{corS2} shows that equality holds in the above
inequality. This means that $\{\mu_n\}$ is exactly the set of all nonzero
eigenvalues of $\left|T^{(s,t)}\right|$, and the eigenspace corresponding
to each $\mu_n$ is one-dimensional.
\end{proof}

Recall that for any bounded linear operator $T$ on a Hilbert space we 
always have $\|T\|^2=\|T^*T\|=\|TT^*\|$. Consequently, if $T$ is
compact, then $\|T\|$ is equal to the first (largest) singular value of
$T$. Replacing $\gamma$ by the formula for $\gamma_+$ from the 
proof Lemma~\ref{existenceofgamma}, we obtain the formula for
$\mu_0$ in Theorem A. Finally, a simple calculation with the help of
the geometric series completes the proof of Theorem A.

\section{Canonical integral operators on $F^p$}

In this section, we consider the boundedness and compactness of the
canonical integral operators $T^{(s,t)}$ on $F^p$ for $0< p\leq\infty$.

\begin{proposition}\label{propKp}
Let $0<p<\infty$, $w\in\C$, and $(s, t)\in\C\times\C$ with $s\neq0$. Then
the kernel function $K_w^{(s,t)}$ belongs to $F^p$ if and only if $|s|>|t|$.
In this case,
$$\left\|K^{(s,t)}_w\right\|_{p}=\frac{|s|^{(1/p-1/2)}}
{(|s|^2-|t|^2)^{1/(2p)}}\exp\left[\frac{|{w}|^2}{2(|s|^2-|t|^2)}\right]
\left|\exp\left[\frac{t(|t|^2+1-|s|^2)}{2\overline{s}(|s|^2-|t|^2)}{w^2}
\right]\right|.$$
\end{proposition}

\begin{proof}
By an argument similar to that given for the case $p =2$ (see Lemma 1 of
\cite{DZ3}) we get $K_w^{(s,t)}\in F^p$ if and only if $|s|>|t|$. Moreover,
for any fixed $w\in\C$ and $|s|>|t|$, it follows from \eqref{eqIC} that
\begin{align*}
\left\|K^{(s,t)}_w\right\|_p^p&=\frac{p}{2\pi{|s|}^{p/2}}\left|\exp
\left(-\frac{p\overline{t}}{2s}\overline{w}^2\right)\right|
\inc\left|\exp\left[\frac{pt}{4s}z^2+\frac{p z\overline{w}}{2s}
\right]\right|^2e^{-\frac{p|z|^2}{2}}\,dA(z)\\[8pt]
&=\frac{1}{{|s|}^{p/2}}\left|\exp\left(-\frac{p\overline{t}}{2s}
\overline{w}^2\right)\right|\inc\left|\exp\left[\frac{{t}}{2s}{u}^2+
\frac{\sqrt{p}u\overline{w}}{\sqrt{2}s}\right]\right|^2d\lambda(u)\\[8pt]
&=\frac{|s|}{{|s|}^{p/2}\sqrt{|s|^2-|t|^2}}\exp\left[\frac{p|{w}|^2}
{2(|s|^2-|t|^2)}\right]\left|\exp\left[\frac{pt(|t|^2+1-|s|^2)}
{2\overline{s}(|s|^2-|t|^2)}{w^2}\right]\right|.
\end{align*}
This proves the desired result.
\end{proof}

\begin{proposition}\label{propKinfty}
Suppose $w\in\C$ and $(s, t)\in\C\times\C$ with $s\neq0$.
\begin{enumerate}
\item[(a)] If $w=0$, then $K_w^{(s,t)}\in F^\infty$ if and only if $|s|\geq|t|$.
\item[(b)] If $w\neq0$, then $K_w^{(s,t)}\in F^\infty$ if and only if $|s|>|t|$.
\end{enumerate}
In both cases,
$$\left\|K^{(s,t)}_w\right\|_\infty\geq\frac{1}{\sqrt{|s|}}
\left|\exp \left[\frac{(\varepsilon^2|s|^2-1)stw^2}{2|s|^2}\right]\right|
\exp\left[\frac{2\varepsilon-\varepsilon^2|s|^2}{2}|w|^2\right]$$
for any $\varepsilon\geq0$.
\end{proposition}

\begin{proof}
If $t=0$, then it is clear that both $|s|>|t|$ and $K_w^{(s,t)}\in F^\infty$
are true for any $w\in\C$. So we assume $t\neq0$ for the rest of the proof.
It is clear that for $w=0$, $K_w^{(s,t)}(z)=e^{tz^2/(2s)}\in F^\infty$ if and
only if the type $|t|/(2|s|)$ is less than or equal to $1/2$, or $|s|\geq|t|$.
For $w\neq0$ and $|s|=|t|$, we easily check that
$$\sup\limits_{z\in\C}|K_w^{(s,t)}(z)|\exp [-|z|^2/2]=\infty,$$
from which the desired result follows.

Recall that, for any $u\in\C$, the Weyl operator $W_u$ on the Fock spaces
is defined by
$$W_uf(z)=k_u(z)f(z-u),\qquad f\in F^p.$$
It is known that $W_u$ is a surjective isometry on $F^{\infty}$. See
Proposition 2.38 of \cite{Zhu1}. Thus for any $\varepsilon\geq0$, we have
\begin{align*}
\left\|K^{(s,t)}_u\right\|_\infty&=\left\|W_uK^{(s,t)}_u\right\|_\infty\\[8pt]
&\geq\left|k_u(\varepsilon su+u)\right|\left|K_u^{(s,t)}(\varepsilon su)\right|
\exp\left[-|\varepsilon su+u|^2/2\right]\\[8pt]
&=\frac{1}{\sqrt{|s|}}\left|\exp \left[\frac{\varepsilon^2|s|^2stu^2-
\overline{st}\overline{u}^2}{2|s|^2}+\frac{\varepsilon s-
\varepsilon\overline{s}}{2}|u|^2\right]\right|\exp\left[
\frac{2\varepsilon-\varepsilon^2|s|^2}{2}|u|^2\right]\\[8pt]
&=\frac{1}{\sqrt{|s|}}\left|\exp\left[\frac{(\varepsilon^2|s|^2-1)stu^2}
{2|s|^2}\right]\right|\exp \left[\frac{2\varepsilon-\varepsilon^2|s|^2}{2}
|u|^2\right].
\end{align*}
The proof is complete.
\end{proof}

As a consequence of the above proposition, we have the following interesting
result which is critical to our proof of the main theorem in this section.

\begin{corollary}\label{corKinftymax}
Let $(s, t)\in\C\times\C$ with $|s|>|t|$. For any $w\in\C$ with
$stw^2=|s||t||w|^2$ we have
$$\exp\left[-\frac{1}{2}|w|^2\right]\left\|K^{(s,t)}_w\right\|_\infty\geq
\frac{1}{\sqrt{|s|}}\exp \left[\frac{1+|t|^2-|s|^2}{2|s|(|s|-|t|)}|w|^2\right].$$
\end{corollary}

\begin{proof}
Consider the function
$$f(x)=\frac{(x^2|s|^2-1)|t|}{2|s|}+\frac{2x-x^2|s|^2-1}{2}$$
for $x\in[0,\infty)$. A few lines of elementary calculations show that
$$\max_{x\in[0,\infty)}|f(x)|=f\left(\frac{1}{|s|(|s|-|t|)}\right)
=\frac{1+|t|^2-|s|^2}{2|s|(|s|-|t|)}.$$
The desired result then follows from Proposition~\ref{propKinfty}.
\end{proof}

As an important application of the upper bound for the two-parameter
integral kernels obtained in Section 2, we can find an upper bound for
the sup-norm of them, which is also critical for our proof of the main 
theorem in this section.

\begin{corollary}\label{corSKinftymax}
Let $w\in\C$ and $(s, t)\in\C\times\C$ with $|s|^2\ge|t|^2+1$. Then
$$\exp\left[-\frac{1}{2}|w|^2\right]\left\|K^{(s,t)}_w\right\|_\infty
\leq C\, \exp\left[-\frac{(|s|^2-|t|^2-1)(|s|^2-1)}
{2|s|^2(3|s|^2-|t|^2+1)}|w|^2\right].$$
Here $C$ is the constant from Theorem~\ref{thmKzw}.
\end{corollary}

\begin{proof}
By Theorem \ref{thmKzw}, we have
\begin{multline*}
\left|K_w^{(s,t)}(z)\right|\exp \left[-\frac{|z|^2+|w|^2}{2}\right]\\
\leq C\,\exp\left[-\frac{(|s|^2-|t|^2-1)}{2|s|^2(3|s|^2-|t|^2+1)}
\left[|s|^2(|z|^2+|w|^2)+2\re ({s\overline{z}w})\right]\right].
\end{multline*}
Since $|s|^2\ge|t|^2+1$, we obtain
\begin{align*}
\exp\left[-\frac{|w|^2}{2}\right]\left\|K^{(s,t)}_w\right\|_\infty
&=\sup\limits_{z\in\C}\left|K_w^{(s,t)}(z)\right|
\exp\left[-\frac{|z|^2+|w|^2}{2}\right]\\
&\leq C\, \sup\limits_{z\in\C}\exp\left[-\frac{(|s|^2-|t|^2-1)
\left[|\overline{s}z+w|^2+(|s|^2-1)|w|^2\right]}
{2|s|^2(3|s|^2-|t|^2+1)}\right]\\
&= C\,\exp\left[-\frac{(|s|^2-|t|^2-1)(|s|^2-1)}
{2|s|^2(3|s|^2-|t|^2+1)}|w|^2\right].
\end{align*}
This proves the desired result.
\end{proof}

\begin{proposition}\label{propsup}
Let $(s, t)\in\C\times\C$ with $|s|^2\ge|t|^2+1$. Then
$$\sup_{w\in\C}\inc \left| K^{(s,t)}(z,w)\right|
e^{-\frac{|z|^2+|w|^2}{2}} dA(z)<\infty.$$
If $|s|^2>|t|^2+1$, then
$$\lim_{|w|\rightarrow\infty}\inc\left| K^{(s,t)}(z,w)\right|
e^{-\frac{|z|^2+|w|^2}{2}} dA(z)=0.$$
\end{proposition}

\begin{proof}
For $|s|^2>|t|^2+1$ we have
\begin{align*}
&\left|\exp\left[\frac{t(1+|t|^2-|s|^2)}{2\overline{s}(|s|^2-|t|^2)}
w^2\right]\right|\exp\left[-\frac{(|s|^2-|t|^2-1)}
{2(|s|^2-|t|^2)}|w|^2\right]\\[8pt]
&\qquad\leq\exp\left[\frac{(|s|^2-|t|^2-1)|t|}{2|s|(|s|^2-|t|^2)}
|w|^2-\frac{(|s|^2-|t|^2-1)}{2(|s|^2-|t|^2)}|w|^2\right]\\[8pt]
&\qquad\leq\exp\left[-\frac{|s|^2-|t|^2-1}{2|s|(|s|+|t|)}|w|^2\right]
\end{align*}
for any $w\in\C$. It follows from Proposition~\ref{propKp} that
\begin{align}\label{IeqIK}
\inc\left| K^{(s,t)}(z,w)\right|e^{-\frac{|z|^2+|w|^2}{2}}dA(z)
&=2\pi\exp\left[-\frac{1}{2}|w|^2\right]\left\|K^{(s,t)}_w
\right\|_1\nonumber\\[8pt]
&\leq \frac{2\pi\sqrt{|s|}}{\sqrt{|s|^2-|t|^2}}
\exp\left[-\frac{|s|^2-|t|^2-1}{2|s|(|s|+|t|)}|w|^2\right].
\end{align}
The desired result then follows immediately.
\end{proof}

\begin{corollary}\label{corsup}
Let $(s, t)\in\C\times\C$ with $|s|^2\ge|t|^2+1$. Then
$$c_1=:\sup_{w\in\C}\inc\left|\langle T^{(s,t)}k_w, k_z\rangle
\right|dA(z)\leq\frac{2\pi\sqrt{|s|}}{\sqrt{|s|^2-|t|^2}}$$
and
$$c_2=:\sup_{w\in\C}\inc\left|\left\langle\left(T^{(s,t)}
\right)^{*}k_w, k_z\right\rangle\right| dA(z)\leq
\frac{2\pi\sqrt{|s|}}{\sqrt{|s|^2-|t|^2}}.$$
\end{corollary}

\begin{proof}
It is clear from Proposition~\ref{propBT} and \eqref{eqTstar} that
$$\inc\left|\left\langle T^{(s,t)}k_w, k_z\right\rangle\right|dA(z)
=\inc \left|K^{({s},t)}(z,w)\right|e^{-\frac{|z|^2+|w|^2}{2}}dA(z),$$
and
\begin{align*}
\inc\left|\left\langle \left(T^{(s,t)}\right)^{*}k_w, k_z\right\rangle\right|dA(z)
&=\inc\left|\left\langle T^{(\overline{s},-t)}k_w, k_z\right\rangle\right|dA(z)\\
&=\inc\left|K^{(\overline{s},-t)}(z,w)\right|e^{-\frac{|z|^2+|w|^2}{2}}dA(z).
\end{align*}
The desired result then follows from \eqref{IeqIK}.
\end{proof}

We are now ready to determine the boundedness and compactness of the
canonical integral operators $T^{(s,t)}$ on $F^p$ for any $0< p\leq\infty$.
First note that for $|t|\ge2|s|$, the integral for $T^{(s,t)}f$ is divergent for
all nonzero polynomials $f$ and for all finite linear combinations of (ordinary)
kernel functions. For $|s|\le|t|<2|s|$, although we have
$$f=K_w\in F^p, \qquad w\neq0,$$
for $0< p\leq\infty$, it follows from \eqref{eqIC} and
Propositions~\ref{propKp} and \ref{propKinfty} that
$$T^{(s,t)}f(z)=K_w^{(s,t)}(z)\notin F^p.$$
This proves that $T^{(s,t)}$ is a densely defined but unbounded linear
operator on $F^p$ when $|s|\le|t|<2|s|$. Therefore, when considering
$T^{(s,t)}$ as operators on $F^p$, we will make the natural assumption
that $|t|<|s|$.

\begin{theorem}\label{thmTBp}
Suppose $|t|<|s|$ and $f\in F^\infty$. Then the integral
$$T^{(s,t)}f(z)=\inc K^{(s,t)}(z,w)f(w)\,d\lambda(w)$$
converges for every $z\in\C$ and $T^{(s,t)}f$ is an entire function.
Moreover, for any $0<p\le\infty$,
\begin{enumerate}
\item[(a)] the operator $T^{(s, t)}$ is bounded on
$F^p$ if and only if $|s|^2\ge|t|^2+1$;
\item[(b)] the operator $T^{(s,t)}$ is compact on $F^p$
if and only if $|s|^2>|t|^2+1$.
\end{enumerate}
\end{theorem}

\begin{proof}
To prove (a), we first assume $|s|^2<|t|^2+1$. Recall that the normalized
reproducing kernels $k_w$ are unit vectors not only in $F^2$ but also in
$F^p$ for $0<p\le\infty$. For $0<p<\infty$, it follows from \eqref{eqTKw} 
and Proposition~\ref{propKp} that
\begin{align*}
\left\|T^{(s,t)}k_w\right\|_p&=\exp\left[-\frac{1}{2}|w|^2\right]
\left\|K^{(s,t)}_w\right\|_p\\[8pt]
&=\frac{|s|^{(1/p-1/2)}}{(|s|^2-|t|^2)^{1/(2p)}}\exp\left[
\frac{|t|^2+1-|s|^2}{2(|s|^2-|t|^2)}|{w}|^2\right]\left|\exp\left
[\frac{t(|t|^2+1-|s|^2)}{2\overline{s}(|s|^2-|t|^2)}{w^2}\right]\right|\\[8pt]
&\geq \frac{|s|^{(1/p-1/2)}}{(|s|^2-|t|^2)^{1/(2p)}}\exp
\left[\frac{1+|t|^2-|s|^2}{2|s|(|s|+|t|)}|w|^2\right]\rightarrow\infty
\end{align*}
as $|w|\rightarrow\infty$, so $T^{(s,t)}$ is unbounded on $F^p$.
When $p=\infty$, we let $\{w_n\}$ be a sequence of nonzero complex
numbers such that $stw_n^2=|stw_n^2|$ and such that the sequence
$\{|w_n|\}$ is nondecreasing and tends to $\infty$. It follows from
\eqref{eqTKw} and Corollary~\ref{corKinftymax} that
$$\|T^{(s,t)}k_{w_n}\|_\infty\geq\frac{1}{\sqrt{|s|}}\exp
\left[\frac{1+|t|^2-|s|^2}{2|s|(|s|-|t|)}|w_n|^2\right]\rightarrow\infty$$
as $n\rightarrow\infty$, so $T^{(s,t)}$ is unbounded on $F^\infty$.

Next we assume $|s|^2\ge|t|^2+1$. For $1< p< \infty$ we let
$1/p+1/q=1$ and let $c_1$, $c_2$ be the constants defined in
Corollary~\ref{corsup}. It is clear from \eqref{eq3} and the proof of
Corollary~\ref{corsup} that
$$\sup_{z\in\C}\inc\left|K^{(s,t)}(z,w)\right|
e^{-\frac{|z|^2+|w|^2}{2}}dA(w)=c_2.$$
By H\"{o}lder's inequality and Proposition~\ref{propsup}, we have
\begin{align*}
&\left|T^{(s,t)}f(z)\right|\\
&\leq\frac{1}{\pi}\inc\left|K^{(s,t)}(z,w)\right||f(w)|e^{-|w|^2}dA(w)\\
&\leq\frac{1}{\pi}\left[\inc\left|K^{(s,t)}(z,w)\right|
e^{-\frac{|w|^2}{2}}dA(w)\right]^{\frac{1}{q}}
\left[\inc\left|K^{(s,t)}(z,w)\right|e^{-\frac{(p+1)|w|^2}{2}}
|f(w)|^pdA(w)\right]^{\frac{1}{p}}\\
&\leq\frac{1}{\pi}\left(c_2 \,e^{\frac{|z|^2}{2}}\right)^{\frac{1}{q}}
\left[\inc\left|K^{(s,t)}(z,w)\right|e^{-\frac{(p+1)|w|^2}{2}}|f(w)|^p
dA(w)\right]^{\frac{1}{p}}
\end{align*}
for any $f\in F^p$. Using Proposition~\ref{propsup} again, we get
\begin{align*}
\|T^{(s, t)}f\|_p^p&=\frac{p}{2\pi}\inc\left|T^{(s,t)}f(z)\right|^p
e^{-\frac{p|z|^2}{2}}dA(z)\\[8pt]
&\leq\left(\frac{1}{\pi}\right)^pc_2^{p/q}\frac{p}{2\pi}
\inc|f(w)|^pe^{-\frac{p|w|^2}{2}}dA(w)
\inc \left|K^{(s,t)}(z,w)\right|e^{-\frac{|z|^2+|w|^2}{2}}dA(z)\\[8pt]
&\leq \left(\frac{1}{\pi}\right)^pc_2^{p/q}c_1\|f\|_p^p,
\end{align*}
which shows that $T^{(s, t)}$ is bounded on $F^p$ with
$$ \|T^{(s, t)}\|\leq \frac{c_1^{1/p}c_2^{1/q}}{\pi}
\leq \frac{2\sqrt{|s|}}{\sqrt{|s|^2-|t|^2}}.$$

When $p=1$ or $p=\infty$, it follows immediately from Fubini's
theorem and Proposition~\ref{propsup} that
$$\|T^{(s, t)}f\|_1\leq \frac{1}{\pi}\|f\|_{1}\sup_{w\in\C}
\inc\left|K^{(s,t)}(z,w)\right|e^{-\frac{|z|^2+|w|^2}{2}}dA(z)
\leq \frac{2\sqrt{|s|}}{\sqrt{|s|^2-|t|^2}}\|f\|_1,$$
and
$$\|T^{(s, t)}f\|_{\infty}\leq\frac{1}{\pi}\|f\|_{\infty}
\sup_{z\in\C}\inc\left|K^{({s},t)}(z,w)\right|
e^{-\frac{|z|^2+|w|^2}{2}}dA(w)
\leq \frac{2\sqrt{|s|}}{\sqrt{|s|^2-|t|^2}}\|f\|_{\infty},$$
which shows that $T^{(s,t)}$ is bounded on $F^1$ and $F^\infty$.

When $0<p<1$, we use Theorem 2.34 of \cite{Zhu1} to represent
$f\in F^p$ as follows:
$$f(z)=\sum_{w\in r\mathbb{Z}^2} c_w k_w(z),$$
where $r\mathbb{Z}^2$ is a square lattice in the complex plane for
sufficiently small $r>0$, $\{c_w : w\in r\mathbb{Z}^2\}\in l^p$,
the series converges in the ``norm topology'' of $F^p$, and
$$\sum_{w\in r\mathbb{Z}^2} |c_w|^p\leq C \|f\|_p^p$$
for some positive constant $C$ that is independent of $f$. By a
standard approximation argument, we have
$$T^{(s,t)}f(z)=\sum_{w\in r\mathbb{Z}^2} c_w T^{(s,t)}k_w(z).$$
Since $|s|^2\ge|t|^2+1$, a straightforward calculation using
\eqref{eqTKw} and Proposition~\ref{propKp} shows that
\begin{equation}\label{eqTkpless}
\left\|T^{(s,t)}k_w\right\|_p^p\leq \frac{|s|^{(2-p)/2}}{\sqrt{|s|^2-|t|^2}}
\exp\left[-\frac{p(|s|^2-|t|^2-1)}{2|s|(|s|+|t|)}|{w}|^2\right]
\end{equation}
for all $w\in\C$.
By H\"{o}lder's inequality, we have
$$\left\|T^{(s,t)}f\right\|_p^p\leq \sum_{w\in r\mathbb{Z}^2}|c_w|^p
\left\|T^{(s,t)}k_w\right\|_p^p
\leq \frac{C|s|^{(2-p)/2}}{\sqrt{|s|^2-|t|^2}}\|f\|_p^p.$$
This proves that $T^{(s,t)}$ is also bounded on $F^p$ when $0<p<1$.
This completes the proof of part (a).

To prove (b) we only need to consider two cases: $|s|^2>|t|^2+1$ and
$|s|^2=|t|^2+1$.

We first consider the case $|s|^2>|t|^2+1$. Let $\{f_n\}$ be a
sequence in $F^p$ such that $\sup\limits_n\{\|f_n\|_{p}\}<\infty$
and $f_n\rightarrow 0$ uniformly on the compact sets in $\C$.
When $1\leq p<\infty$, by the proof in part (a), we get
\begin{align*}
\|T^{(s, t)}f_n\|_p^p &=\frac{p}{2\pi}\inc \left|T^{(s,t)}f_n(z)
\right|^pe^{-\frac{p|z|^2}{2}}dA(z)\\
&\leq C\,\inc|f_n(w)|^pe^{-\frac{p|w|^2}{2}}dA(w)
\inc\left|K^{(s,t)}(z,w)\right|e^{-\frac{|z|^2+|w|^2}{2}}dA(z)
\end{align*}
for some constant $C$. Let
$$D_r=\{w\in\C: |w|\leq r\},\qquad 0<r<\infty.$$
Then
\begin{align*}
\|T^{(s, t)}f_n\|_p^p&\leq C\,\int_{\C\setminus D_r}|f_n(w)|^p
e^{-\frac{p|w|^2}{2}}dA(w)\inc\left|K^{(s,t)}(z,w)\right|
e^{-\frac{|z|^2+|w|^2}{2}}dA(z)\\
&\qquad+C\int_{D_r} |f_n(w)|^pe^{-\frac{p|w|^2}{2}}dA(w)
\inc \left|K^{(s,t)}(z,w)\right|e^{-\frac{|z|^2+|w|^2}{2}}dA(z)\\
&= I_1+I_2.
\end{align*}
Since $\sup\limits_n\{\|f_n\|_{p}\}<\infty$, it follows from
Proposition~\ref{propsup} that
\begin{align*}
I_1&= C\int_{\C\setminus D_r} |f_n(w)|^pe^{-\frac{p|w|^2}{2}}dA(w)
\inc \left|K^{(s,t)}(z,w)\right|e^{-\frac{|z|^2+|w|^2}{2}}dA(z)\\
&\leq \frac{2\pi C}{p}\,\|f_n\|_{p}^p\sup_{w\in \C\setminus D_r}
\inc\left|K^{(s,t)}(z,w)\right|e^{-\frac{|z|^2+|w|^2}{2}}dA(z)\rightarrow 0
\end{align*}
uniformly in $n$ as $r\rightarrow\infty$. Also, since $f_n\rightarrow 0$
uniformly on compact sets in $\C$, we have
\begin{align*}
I_2&=C\int_{D_r}|f_n(w)|^pe^{-\frac{p|w|^2}{2}}dA(w)
\inc\left|K^{(s,t)}(z,w)\right|e^{-\frac{|z|^2+|w|^2}{2}}dA(z)\\
&\leq\frac{2\pi C}{p}\,\sup_{w\in D_r}|f_n(w)|^p\,\sup_{w\in \C}
\inc\left|K^{(s,t)}(z,w)\right|e^{-\frac{|z|^2+|w|^2}{2}}dA(z)\rightarrow 0
\end{align*}
as $n\rightarrow\infty$. So $\|T^{(s, t)}f_n\|_p\rightarrow 0$ as
$n\rightarrow\infty$ for any $1\leq p<\infty$.

For $p=\infty$, it is clear that
\begin{align*}
    \|T^{(s, t)}f_n\|_\infty&\leq \frac{1}{\pi}\,\sup_{z\in\C}
\left\{\inc |f_n(w)|e^{-\frac{|w|^2}{2}}\left|K^{({s},t)}(z,w)
\right|e^{-\frac{|z|^2+|w|^2}{2}}dA(w)\right\}\\
&\leq \frac{1}{\pi}\, \inc |f_n(w)|e^{-\frac{|w|^2}{2}} 
\left\|K^{(s,t)}_w\right\|_\infty e^{-\frac{|w|^2}{2}} dA(w).
\end{align*}
Using Corollary \ref{corSKinftymax}, we obtain
$$\lim_{|w|\rightarrow\infty}\ \left\|K^{(s,t)}_w\right\|_\infty 
e^{-\frac{|w|^2}{2}}=0.$$
Then a similar argument used above shows that
$\|T^{(s, t)}f_n\|_\infty\rightarrow 0$ as $n\rightarrow\infty$.
Thus $T^{(s,t)}$ is compact on $F^p$, $1\le p\le\infty$, when
$|s|^2>|t|^2+1$.

If $0<p<1$ and $|s|^2>|t|^2+1$, it follows from \eqref{eqTkpless} that
$$\lim_{|w|\rightarrow\infty}\left\|T^{(s,t)}k_w\right\|_p=0.$$
Using H\"{o}lder's inequality as in the proof of part (a) and then
proceeding as in the previous paragraph, we see that
$\|T^{(s, t)}f_n\|_p\rightarrow 0$ as $n\rightarrow\infty$.
Thus $T^{(s,t)}$ is compact on $F^p$ for all $0<p\le\infty$ when
$|s|^2>|t|^2+1$.

Next we consider the case $|s|^2=|t|^2+1$. If $0< p<\infty$, it
follows from \eqref{eqTKw} and Proposition~\ref{propKp} that
$$\|T^{(s,t)}k_w\|_p=|s|^{(1/p-1/2)}\nrightarrow 0,
\qquad |w|\rightarrow\infty.$$
When $p=\infty$, we let $\{w_n\}$ be a sequence of nonzero
complex numbers such that $stw_n^2=|stw_n^2|$ and such that
the sequence $\{|w_n|\}$ is nondecreasing and tends to $\infty$.
Then it follows from Corollary~\ref{corKinftymax} that
$$\|T^{(s,t)}k_{w_n}\|_\infty\geq \frac{1}{\sqrt{|s|}}\nrightarrow 0,
\qquad n\rightarrow\infty.$$
Note that each $k_w$ is a unit vector in $F^p$ and $k_w(z)\to0$
(as $|w|\rightarrow\infty$) uniformly for $z$ in any compact subset
of $\C$ . Thus $T^{(s,t)}$ is not compact on $F^p$ for
any $0<p\le\infty$ when $|s|^2=|t|^2+1$. This proves part (b) and
completes the proof of the theorem.
\end{proof}

We note that Theorems 2.4 and 33 in \cite{LZZ} are related to our
results here in the case $0<p<1$.


\begin{thebibliography}{99}
\bibitem{B1} V. Bargmann, On a Hilbert space of analytic functions
and an associated integral transform I, \textit{Comm. Pure Appl. Math.}
\textbf{14}, 187-214 (1961).
\bibitem{DZ3} X. Dong and K. Zhu,  Canonical integral operators on
the Fock space, \textit{Math. Z.} \textbf{306}, 64 (2024).
\bibitem{HKOS} J. Healy, M. Kutay, H. Ozaktas, and J. Sheridan, 
\textit{Linear Canonical Transforms Theory and Applications}, 
Springer, New York, 2016.
\bibitem{LZZ} Z. Lou, K. Zhu, and S. Zhu, Linear operators on Fock spaces,
\textit{Integral Equations Operator Theory} \textbf{88}, 287-300 (2017).
\bibitem{Zhu2} K. Zhu, \textit{Operator Theory in Function Spaces}
(2nd edition), Mathematical Surveys and Monographs {\bf 138},
American Mathematical Society, 2007.
\bibitem{Zhu1} K. Zhu, \textit{Analysis on Fock Spaces},
Springer GTM {\bf 263}, New York, 2012.
\end{thebibliography}
\end{document}